\documentclass[12pt, reqno]{amsart}

\usepackage{amsmath, amsthm, amscd, amsfonts, amssymb, graphicx, color}
\usepackage[bookmarksnumbered, colorlinks, plainpages]{hyperref}

\newtheorem{theorem}{Theorem}[section]
\newtheorem{lemma}[theorem]{Lemma}

\theoremstyle{definition}

\theoremstyle{remark}
\newtheorem{remark}[theorem]{Remark}

\begin{document}
\setcounter{page}{1}

\title[Certain Inequalities Involving the $q$-Deformed Gamma Function ]{Certain Inequalities Involving the $q$-Deformed Gamma Function$^{\dagger}$ }

\author[Kwara Nantomah and Edward Prempeh]{Kwara Nantomah$^{*}$$^1$  and Edward Prempeh$^2$}

\address{$^{1}$ Department of Mathematics, University for Development Studies, Navrongo Campus, P. O. Box 24, Navrongo, UE/R, Ghana. }
\email{\textcolor[rgb]{0.00,0.00,0.84}{mykwarasoft@yahoo.com, knantomah@uds.edu.gh}}

\address{$^{2}$ Department of Mathematics, Kwame Nkrumah University of Science and Technology, Kumasi, Ghana. }
\email{\textcolor[rgb]{0.00,0.00,0.84}{eprempeh.cos@knust.edu.gh}}


\let\thefootnote\relax\footnotetext{$^{*}$ Corresponding author.
\newline \indent $^{\dagger}$ Cite this article as: "K. Nantomah and E. Prempeh, \textit{Certain inequalities involving the $q$-deformed Gamma function}, Problemy Analiza - Issues of Analysis, \textbf{3}(22)(2015), No.1, In press."}

\begin{abstract}
This paper is inspired by the work of J. S\'{a}ndor in 2006. In the paper, the authors establish some double inequalities involving  the ratio $ \frac{\Gamma_{q}(x+1)}{ \Gamma_{q} \left( x+\frac{1}{2}\right)}$, where $\Gamma_{q}(x)$ is  the $q$-deformation of the classical Gamma function denoted by $\Gamma(x)$. The method employed in presenting the results makes use of Jackson's $q$-integral representation of the $q$-deformed Gamma function. In addition, H\"{o}lder's inequality for  the $q$-integral, as well as  some basic analytical techniques involving the $q$-analogue of the psi function are used. As a consequence, $q$-analogues of the classical Wendel's asymptotic relation are obtained. At the end, sharpness of the inequalities established in this paper is investigated.\\


\begin{center}  \emph{Submitted: 5 November 2014, Accepted: 17 March 2015} \end{center}

\end{abstract} 

\maketitle

\section{Introduction and Preliminaries}
\noindent
Let $\Gamma(x)$ be the well-known classical Gamma function  defined for $x>0$ by
\begin{equation*}\label{eqn:gamma}
\Gamma(x)=\int_0^\infty t^{x-1}e^{-t}\,dt. 
\end{equation*}

\noindent
The psi function $\psi(x)$ otherwise known as the digamma function is defined as the logarithmic derivative of the Gamma function. That is,
\begin{equation*}\label{eqn:Psi}
\psi(x)=\frac{d}{dx}\ln(\Gamma(x))=\frac{\Gamma'(x)}{\Gamma(x)},\qquad x>0.
\end{equation*}

\noindent
The Jackson's $q$-integral from $0$ to $a$ and from $0$ to $\infty$ are defined as follows  \cite{Jackson-1910}:
\begin{equation*}\label{eqn:q-integral-1}
\int_0^{a}f(t)\,d_{q}t=(1-q)a\sum_{n=0}^{\infty}f(aq^{n})q^{n},
\end{equation*}

\begin{equation*}\label{eqn:q-integral-2}
\int_0^{\infty}f(t)\,d_{q}t=(1-q)\sum_{-\infty}^{\infty}f(q^{n})q^{n}
\end{equation*}
provided that the sums converge absolutely.\\

\noindent
In a generic interval $[a,b]$, the Jackson's $q$-integral takes the following form:
\begin{equation*}\label{eqn:q-integral-generic}
\int_a^{b}f(t)\,d_{q}t = \int_0^{b}f(t)\,d_{q}t  - \int_0^{a}f(t)\,d_{q}t .
\end{equation*}

\noindent
For more information on this special integral, see \cite{Jackson-1910}.\\

\noindent
For $a\in C$, the set of complex numbers, we have  the following notations:\\
$(a;q)_0=1$, \, $(a;q)_n=\prod_{i=0}^{n-1}(1-aq^i)$, \, $(a;q)_{\infty}=\prod_{i=0}^{\infty}(1-aq^i)$\, and \, $[n]_{q}!=\frac{(q;q)_n}{(1-q)^n}$.\\

\noindent
The $q$-deformed Gamma function (also known as the $q$-Gamma function or the $q$-analogue of the Gamma function) is defined for $q\in(0,1)$ and $x>0$  by
\begin{align}
\Gamma_q(x)& =\int_0^{\frac{1}{1-q}} t^{x-1}E_{q}^{-qt}\,d_{q}t = \int_0^{[\infty]_q} t^{x-1}E_{q}^{-qt}\,d_{q}t \\
&  = (1-q)^{1-x}\prod_{n=0}^{\infty}\frac{1-q^{n+1}}{1-q^{n+x}} \nonumber
\end{align}
where \, $E_{q}^{t}=\sum_{n=0}^{\infty}q^{\frac{n(n-1)}{2}}\frac{t^{n}}{[n]_{q}!}=(-(1-q)t;q)_{\infty}$\, is a $q$-analogue of the classical exponential function. See also \cite{Askey-1978}, \cite{Chung-Kim-Mansour-2014},  \cite{Ege-Yyldyrym-2012}, \cite{Elmonster-Brahim-Fitouhi-2012} and the references therein. \\

\noindent
The function, $\Gamma_q$ exhibits  the following properties (see  \cite{Chung-Kim-Mansour-2014}),
\begin{align}
\Gamma_{q}(x+1)&=[x]_{q}\Gamma_q(x) \label{eqn:functional-q-Gamma},\\
\Gamma_{q}(1)&=1, \nonumber \\
\Gamma_{q}\left (\frac{1}{2} \right)&=\sqrt{\pi_q} \nonumber 
\end{align}
where \, $[x]_q=\frac{1-q^x}{1-q}$,\,  and \,$\pi_q=q^{\frac{1}{4}} \left(\left[-\frac{1}{2} \right]_{q^{2}}! \right)^2$  is the $q$-analogue of $\pi$. Note that $\pi_q$ is obtained by setting $n=0$ in the $q$-factorial, $[n]_q!$. \\

\noindent
Let $\psi_q(x)$ be the $q$-analogue of the psi function similarly defined for $x>0$ as follows: 
\begin{equation*}\label{eqn:Psi}
\psi_q(x)=\frac{\Gamma_q'(x)}{\Gamma_q(x)}= -\ln(1-q) + \ln q \sum_{n=0}^{\infty}\frac{q^{n+x}}{1-q^{n+x}}
\end{equation*}
It is well-known in literature that this function is increasing for $x>0$. For instance, see Lemma 2.2 of  \cite{Shabani-2008}.\\

\noindent
In 1987, Lew, Frauenthal and Keyfitz \cite{Lew-Frauenthal-Keyfitz-1987} by studying certain problems of traffic 
flow established the double inequality:\\
\begin{equation}\label{eqn:traffic-flow}
2\Gamma \left( n+\frac{1}{2}\right) \leq \Gamma \left( \frac{1}{2}\right)\Gamma(n+1) \leq 2^n \Gamma \left( n+\frac{1}{2}\right), \quad n\in N
\end{equation}
The inequalities ~(\ref{eqn:traffic-flow}) can be rearranged as follows:

\begin{equation*}\label{eqn:traffic-flow-2}
\frac{2}{\sqrt{\pi}} \leq \frac{\Gamma(n+1)}{ \Gamma \left( n+\frac{1}{2}\right)} \leq \frac{2^n}{\sqrt{\pi}}.
\end{equation*}

\noindent
In 2006, S\'{a}ndor \cite{Sandor-2006} by using the following inequalities due Wendel \cite{Wendel-1948}:
\begin{equation}\label{eqn:Wendel}
\left( \frac{x}{x+s}\right)^{1-s} \leq \frac{\Gamma(x+s)}{ x^s \Gamma(x)} \leq 1
\end{equation}
for $x>0$ and $s\in(0,1)$, extended and refined inequality ~(\ref{eqn:traffic-flow}) as follows:

\begin{equation}\label{eqn:Sandor}
\sqrt{x} \leq \frac{\Gamma(x+1)}{ \Gamma \left( x+\frac{1}{2}\right)} \leq \sqrt{x+\frac{1}{2}}
\end{equation}
for $x>0$ .\\

\noindent
The objective of this paper is to establish certain inequalities involving the $q$-deformed Gamma function. First,  employing similar techniques as in  \cite{Sandor-2006},  \cite{Wendel-1948}, and \cite{Qi-Luo-2012}, we prove an $q$-analogue of the double inequality ~(\ref{eqn:Sandor}). Next, using basic analytical procedures, we prove some related double inequality. At the end, we investigate the sharpness of the inequalities established.


\section{Main Results}

\noindent
Let us begin with the following Lemma. 

\begin{lemma}\label{lem:q-of-Wendel}
Assume that  $s\in(0,1)$ and $q\in(0,1)$. Then for any $x>0$ the following inequality is valid.
\begin{equation}\label{eqn:q-of-Wendel} 
\left( \frac{[x]_q}{[x+s]_q}\right)^{1-s} \leq \frac{\Gamma_{q}(x+s)}{ [x]_q^s \Gamma_{q}(x)} \leq 1
\end{equation}
\end{lemma}

\begin{proof}
We employ H\"{o}lder's inequality for Jackson's $q$-integral: 
\begin{equation*}
\int_0^{\infty}f(t)g(t)\,d_qt \leq \left[\int_0^{\infty}(f(t))^a\,d_qt  \right]^{\frac{1}{a}}  \left[\int_0^{\infty}(g(t))^b\,d_qt  \right]^{\frac{1}{b}}
\end{equation*}
where $\frac{1}{a}+\frac{1}{b}=1$ and $a>1$.\\
Let $a=\frac{1}{1-s}$, \, $b=\frac{1}{s}$, \, $f(t)=t^{(1-s)(x-1)}E_{q}^{-(1-s)qt}$\, and \,$g(t)=t^{sx}E_{q}^{-sqt}$\\
Then H\"{o}lder's inequality implies
\begin{align*}
\Gamma_{q}(x+s)&=\int_0^{\frac{1}{1-q}}t^{x+s-1}E_{q}^{-qt}\,d_qt \\
&\leq \left[\int_0^{\frac{1}{1-q}}\left(  t^{(1-s)(x-1)} E_{q}^{-(1-s)qt} \right)^{\frac{1}{1-s}}\,d_qt  \right]^{1-s} \\
&\quad \times \left[\int_0^{\frac{1}{1-q}}\left(  t^{sx} E_{q}^{-sqt} \right)^{\frac{1}{s}}\,d_qt  \right]^{s} \\
&=\left[\int_0^{\frac{1}{1-q}}  t^{x-1} E_{q}^{-qt}\,d_qt  \right]^{1-s}  \left[\int_0^{\frac{1}{1-q}}  t^{x} E_{q}^{-qt}\,d_qt  \right]^{s} \\
&=\left[ \Gamma_{q}(x) \right]^{1-s}\left[ \Gamma_{q}(x+1) \right]^{s}
\end{align*}
Thus,
\begin{equation}\label{eqn:q-Gamma-recursive-1} 
\Gamma_{q}(x+s) \leq \left[ \Gamma_{q}(x) \right]^{1-s}\left[ \Gamma_{q}(x+1) \right]^{s}
\end{equation}
Substituting  ~(\ref{eqn:functional-q-Gamma}) into  ~(\ref{eqn:q-Gamma-recursive-1}) yields
\begin{equation*}\label{eqn:q-Gamma-recursive-2} 
\Gamma_{q}(x+s) \leq \left[ \Gamma_{q}(x) \right]^{1-s}[x]_{q}^{s}\left[ \Gamma_{q}(x) \right]^{s}
\end{equation*}
which implies
\begin{equation}\label{eqn:q-Gamma-recursive-3} 
\Gamma_{q}(x+s) \leq [x]_{q}^{s} \Gamma_{q}(x). 
\end{equation}
Replacing $s$ by $1-s$ in equation ~(\ref{eqn:q-Gamma-recursive-3}) gives
\begin{equation}\label{eqn:q-Gamma-recursive-4} 
\Gamma_{q}(x+1-s) \leq [x]_{q}^{1-s} \Gamma_{q}(x). 
\end{equation}
Substituting $x$ by $x+s$ results to
\begin{equation}\label{eqn:q-Gamma-recursive-5} 
\Gamma_{q}(x+1) \leq [x+s]_{q}^{1-s} \Gamma_{q}(x+s). 
\end{equation}
Now combining inequalities ~(\ref{eqn:q-Gamma-recursive-3}) and ~(\ref{eqn:q-Gamma-recursive-5}) gives
\begin{equation*}\label{eqn:q-Gamma-double-ineq-1} 
\frac{\Gamma_{q}(x+1) }{[x+s]_{q}^{1-s}}\leq  \Gamma_{q}(x+s) \leq [x]_q^{s}\Gamma_{q}(x)
\end{equation*}
which can be written as
\begin{equation}\label{eqn:q-Gamma-double-ineq-2} 
\frac{[x]_q}{[x+s]_{q}^{1-s}}\Gamma_{q}(x) \leq  \Gamma_{q}(x+s) \leq [x]_q^{s}\Gamma_{q}(x).
\end{equation}
Finally, equation ~(\ref{eqn:q-Gamma-double-ineq-2}) can be rearranged as:
\begin{equation*}
\left( \frac{[x]_q}{[x+s]_q}\right)^{1-s} \leq \frac{\Gamma_{q}(x+s)}{ [x]_q^s \Gamma_{q}(x)} \leq 1
\end{equation*}
concluding the proof of Lemma \ref{lem:q-of-Wendel}.
\end{proof}

\begin{theorem}\label{thm:q-of-Sandor}
Assume that $q\in(0,1)$. Then the  inequality
\begin{equation}\label{eqn:q-of-Sandor} 
\sqrt{[x]_q} \leq \frac{\Gamma_{q}(x+1)}{ \Gamma_{q} \left( x+\frac{1}{2}\right)} \leq \sqrt{\left[x+\frac{1}{2}\right]_{q}}
\end{equation}
is valid  for any $x>0$.
\end{theorem}

\begin{proof}
By setting $s=\frac{1}{2}$ in the $q$-analogue of Wendel's inequalities ~(\ref{eqn:q-of-Wendel}), we get
\begin{equation}\label{eqn:q-of-Sandor-v1} 
\frac{1}{\sqrt{[x]_q}} \leq \frac{\Gamma_{q}(x)}{ \Gamma_{q} \left( x+\frac{1}{2}\right)} \leq \frac{\sqrt{\left[x+\frac{1}{2}\right]_{q}}}{[x]_q}
\end{equation}
Using  ~(\ref{eqn:functional-q-Gamma}), we can arrange  ~(\ref{eqn:q-of-Sandor-v1}) as follows:
\begin{equation*}
\sqrt{[x]_q} \leq \frac{\Gamma_{q}(x+1)}{ \Gamma_{q} \left( x+\frac{1}{2}\right)} \leq \sqrt{\left[x+\frac{1}{2}\right]_{q}}.
\end{equation*}
That completes the proof.
\end{proof}

\begin{remark}
Inequalities ~(\ref{eqn:q-of-Wendel})  imply
\begin{equation}\label{eqn:q-limit}
\lim_{x \rightarrow \infty}\frac{\Gamma_{q}(x+s)}{[x]_q^{s}\Gamma_{q}(x)}=1. 
\end{equation}
\end{remark}

\begin{remark}
Since \,
$[x]_q^{\beta-\alpha}\frac{\Gamma_{q}(x+\alpha)}{\Gamma_{q}(x+\beta)} = \frac{\Gamma_{q}(x+\alpha)}{[x]_q^{\alpha}\Gamma_{q}(x)}. \frac{ [x]_q^{\beta} \Gamma_{q}(x)}{\Gamma_{q}(x+\beta)}$,\,   we obtain by  ~(\ref{eqn:q-limit}), 
\begin{equation}\label{eqn:q-limit2}
\lim_{x \rightarrow \infty}[x]_q^{\beta-\alpha}\frac{\Gamma_{q}(x+\alpha)}{\Gamma_{q}(x+\beta)}=1, \quad \alpha, \beta \in(0,1).
\end{equation}
\end{remark}

\begin{remark}
The equalities ~(\ref{eqn:q-limit})  and ~(\ref{eqn:q-limit2})  are the $q$-analogues of the classical Wendel's  asymptotic relation \cite{Wendel-1948}:
\begin{equation}
\lim_{x \rightarrow \infty}\frac{\Gamma(x+s)}{x^{s}\Gamma(x)}=1. 
\end{equation}
\end{remark}

\begin{theorem}\label{thm:q-of-Sandor-v2}
Assume that $q\in(0,1)$ is fixed. Then the  inequality
\begin{equation}\label{eqn:q-of-Sandor-v2} 
\frac{1}{\sqrt{\pi_q}} <  \frac{\Gamma_{q}(x+1)}{ \Gamma_{q} \left( x+\frac{1}{2}\right)} < \left(1+\sqrt{q}\right).\frac{1}{\sqrt{\pi_q}}
\end{equation}
is valid  for $x\in(0,1)$.
\end{theorem}

\begin{proof}
Define a function $U(q,x)$ for $q\in(0,1)$ and $x\geq0$ by
\begin{equation*}
U(q,x)= \frac{\Gamma_{q}(x+1)}{ \Gamma_{q} \left( x+\frac{1}{2}\right)}. 
\end{equation*}
Notice that  $\Gamma_{q}(1)=\Gamma_{q}(2)=1$, \,  $\Gamma_{q}(\frac{1}{2}+1)=[\frac{1}{2}]_q\Gamma_{q}(\frac{1}{2})=[\frac{1}{2}]_q\sqrt{\pi_q}$, \,   $[\frac{1}{2}]_q=\frac{1-\sqrt{q}}{1-q}$, \,  $U(q,0)=\frac{1}{\sqrt{\pi_q}}$ \, and \, $U(q,1)=\left(1+\sqrt{q}\right).\frac{1}{\sqrt{\pi_q}}$. \\
Now let $f(q,x)=\ln U(q,x)$. Then,
\begin{equation*}
f(q,x)=\ln \frac{\Gamma_{q}(x+1)}{ \Gamma_{q} \left( x+\frac{1}{2}\right)}=\ln\Gamma_{q}(x+1) - \ln \Gamma_{q} \left( x+\frac{1}{2}\right). 
\end{equation*}
For a fixed $q\in(0,1)$ we obtain
\begin{equation*}
f'(q,x)=\psi_{q}(x+1) - \psi_{q} \left( x+\frac{1}{2}\right)>0, 
\end{equation*}
since $\psi_q(x)$ is increasing for $x>0$. Hence, $U(q,x)=e^{f(q,x)}$ is increasing on  $x>0$ and for $x\in(0,1)$ we have
\begin{equation*}
U(q,0) < U(q,x) <  U(q,1)
\end{equation*}
establishing  ~(\ref{eqn:q-of-Sandor-v2}).
\end{proof}

\begin{remark}
Define $F$ by  $F(q,x)=[x]_q^{-\frac{1}{2}} \frac{\Gamma_q(x+1)}{\Gamma_q(x+\frac{1}{2})}$\, for $q\in(0,1)$ and $x>0$. Let $g(x) = \ln F(q,x)=\ln \Gamma_q(x+1) - \ln \Gamma_q(x+\frac{1}{2}) - \frac{1}{2}\ln[x]_q$. Then,
$g'(x) = \psi_q(x+1) - \psi_q(x+\frac{1}{2}) + \frac{1}{2}\frac{q^x\ln q}{1-q^x}$. 
The following Stieltjes integral representations are valid.
\begin{equation*}
\psi_q(x)=-\ln(1-q)-\int_0^{\infty}\frac{e^{-xt}}{1-e^{-t}}\,d\mu_q(t), \quad  \int_0^{\infty}e^{-xt}\,d\mu_q(t)=-\frac{q^x\ln q}{1-q^x}
\end{equation*}
where $\mu_q(t) = -\ln q \sum_{k=1}^{\infty}\delta(t+k\ln q)$ and $\delta$ represents the Dirac delta function.
See \cite{Ismail-Muldoon-2013} and the references therein. Then that implies, 
\begin{align*}
g'(x)& =-\int_0^{\infty}\frac{e^{-(x+1)t}}{1-e^{-t}}\,d\mu_q(t) + \int_0^{\infty}\frac{e^{-(x+\frac{1}{2})t}}{1-e^{-t}}\,d\mu_q(t) -  \frac{1}{2}\int_0^{\infty}e^{-xt}\,d\mu_q(t)\\
      &=\int_0^{\infty}\frac{e^{-xt}}{1-e^{-t}}\phi(t)\,d\mu_q(t)
\end{align*}
where $\phi(t)=e^{-\frac{1}{2}t}-\frac{1}{2}e^{-t}-\frac{1}{2}<0$. By the Hausdorff-Bernstein-Widder theorem  (see \cite{Mortici-2010} and the references therein), we obtain $g'(x)<0$,  so $g(x)$ is strictly deacreasing. Consequently, $F(q,x)$ is strictly decreasing. Hence, $F(q,x)\geq \lim_{x\rightarrow \infty}F(q,x) =1$ yielding the lower bound of  ~(\ref{eqn:q-of-Sandor}).
\end{remark}

\begin{remark} 
Define $G$ by  $G(q,x)=[x+\frac{1}{2}]_q^{-\frac{1}{2}} \frac{\Gamma_q(x+1)}{\Gamma_q(x+\frac{1}{2})}$\, for $q\in(0,1)$ and $x>0$. Let $w(x) = \ln G(q,x)=\ln \Gamma_q(x+1) - \ln \Gamma_q(x+\frac{1}{2}) - \frac{1}{2}\ln[x+\frac{1}{2}]_q$. Then,
$w'(x) = \psi_q(x+1) - \psi_q(x+\frac{1}{2}) + \frac{1}{2}\frac{q^{x+\frac{1}{2}}\ln q}{1-q^{x+\frac{1}{2}}}$. 
By setting $a=\frac{1}{2}$, $b=1$, $c=\frac{1}{2}$ and $k=1$ in Theorem 7.2 of \cite{Qi-2010-Hindawi}, we obtain,\\
$\psi_q(x+1) - \psi_q(x+\frac{1}{2}) \geq - \frac{1}{2}\frac{q^{x+\frac{1}{2}}\ln q}{1-q^{x+\frac{1}{2}}}$.\\
That implies $w'(x)\geq0$,  so $w(x)$ is increasing. As a result, $G(q,x)$ is also increasing. Hence, $G(q,x)\leq \lim_{x\rightarrow \infty}G(q,x) =1$ yielding the upper bound of  ~(\ref{eqn:q-of-Sandor}).
\end{remark}

\begin{remark}
Let $H(q,x)=\frac{\sqrt{\pi_q}\Gamma_q(x+1)}{\Gamma_q(x+\frac{1}{2})}$. Then,  $H(q,x)$ is increasing, and  for $x\in(0,1)$ we have, \\ $1=\lim_{x\rightarrow 0^+}H(q,x)\leq H(q,x)$ and $H(q,x)\leq \lim_{x\rightarrow 1^-}H(q,x)=1+\sqrt{q}$ respectively yielding the lower and upper bounds of  ~(\ref{eqn:q-of-Sandor-v2}). 
\end{remark}

\noindent
By the above remarks, the estimates in  ~(\ref{eqn:q-of-Sandor}) and ~(\ref{eqn:q-of-Sandor-v2}) are sharp.\\

\section*{Acknowledgements}
The authors are very grateful to the anonymous refrees  for their useful comments and suggestions, which helped in improving the quality of this paper. \\

\bibliographystyle{plain}


\end{document}